\documentclass[11pt,reqno]{amsart}

\usepackage[a4paper,
            top=30mm,bottom=30mm,
            left=33mm,right=33mm]{geometry}

\usepackage[T1]{fontenc}
\usepackage[unicode=true,hidelinks]{hyperref}
\usepackage{amsmath, amsthm, amsfonts, amssymb, mathtools}

\theoremstyle{plain}
\newtheorem{theorem}{Theorem}[section]
\newtheorem{lemma}[theorem]{Lemma}

\newtheorem{proposition}[theorem]{Proposition}

\theoremstyle{remark}

\numberwithin{equation}{section}

\newcommand{\eps}{\varepsilon}
\newcommand{\dy}{\mathrm{d} y}
\newcommand{\rd}{\mathrm{d}}
\newcommand{\N}{\mathbb{N}}
\newcommand{\R}{\mathbb{R}}
\newcommand{\Q}{\mathbb{Q}}
\newcommand{\Z}{\mathbb{Z}}

\newcommand{\bee}{\begin{equation}}
\newcommand{\eee}{\end{equation}}

\title[]{Lorentz estimates for multilinear \\ convolution operators}

\author[]{Nuno J. Alves}
\author[]{Ana \v{C}olovi\'{c}}
\author[]{Loukas Grafakos}

\address[N. J. Alves]{
      CEMSE Division, King Abdullah University of Science and Technology (KAUST), Thuwal 23955-6900, Saudi Arabia.}
\email{nuno.januarioalves@kaust.edu.sa}

\address[A. \v{C}olovi\'{c}]{
      Department of Mathematics, University of Missouri, Columbia MO 65211, USA.}
\email{acg7y@umsystem.edu}

\address[L. Grafakos]{
      Department of Mathematics, University of Missouri, Columbia MO 65211, USA.}
\email{grafakosl@umsystem.edu}

\begin{document}

\begin{abstract}
We study O'Neil-type inequalities in Lorentz spaces for two extremal geometric configurations of translation-invariant multilinear convolution operators. For the one-parameter model, we complement the existing weak-kernel theory with estimates for kernels in the Lorentz space $L^{r,s}$, obtaining a stronger target Lorentz space when $s$ is finite. For the full-dimensional model, we prove a Lorentz refinement of Oberlin's multilinear Young inequality in the interior of its admissible range. For nonnegative full-dimensional kernels, we extend H\"ormander's necessary condition to the multilinear setting in Lorentz spaces and obtain an additional restriction on the secondary indices in the critical case. We also prove the failure of the full-dimensional weak $L^1$ endpoint, use additive combinatorics to construct one-parameter restricted weak-type counterexamples for $p<1$, and obtain complementary positive results for kernels of special structure.
\end{abstract}

\subjclass[2020]{Primary 42B20; Secondary 46E30}
\keywords{Multilinear convolution, Lorentz spaces, multilinear interpolation, O'Neil's inequality.}
\maketitle
\thispagestyle{empty}

\section{Introduction}

This article is concerned with Lorentz norm estimates for multilinear convolution operators. Recall that for measurable functions $f_1,f_2$ on $\R^n$, the (linear) convolution $f_1*f_2$ is a commutative operation defined pointwise by 
\[
(f_1*f_2)(x)
:=
\int_{\R^n}f_1(x-y)f_2(y)\,\rd y, \qquad x\in \R^n,
\]
provided that the integral converges absolutely for the given $x\in\R^n$.

The classical Young convolution inequality on $\R^n$ states that, for
$1\leq p_1,p_2,p\leq\infty$ satisfying the exponent relation
\bee\label{Young100}
1+\frac1p=\frac1{p_1}+\frac1{p_2},
\eee
one has
\bee\label{99}
\|f_1*f_2\|_p
\leq
\|f_1\|_{p_1}\|f_2\|_{p_2}
\eee
for all $f_1\in L^{p_1}(\R^n)$ and $f_2\in L^{p_2}(\R^n)$. Here and throughout this paper, the $L^p$ norm 
of a function $f$ is denoted by $\|f\|_{p}= ( \int_{\R^n} |f|^p \, dx)^{1/p}$.

Before stating the extension of estimate \eqref{99}  to Lorentz spaces we 
recall the relevant definitions. Let $f^*$ denote the decreasing rearrangement of a
measurable function $f$ on $\R^n$. For $0<p<\infty$ and $0<q\leq\infty$, we define the quantities 
\[
\|f\|_{p,q}
:=
\begin{dcases}
\displaystyle
\left(
\int_0^\infty
\big(t^{1/p}f^*(t)\big)^q
\frac{\rd t}{t}
\right)^{1/q},
& 0<q<\infty,\\[2mm]
\displaystyle
\sup_{t>0}t^{1/p}f^*(t),
& q=\infty.
\end{dcases}
\]
The Lorentz space $L^{p,q}(\R^n)$ consists of all measurable functions $f$ for
which $\|f\|_{p,q}$ is finite. In particular, $L^{p,p}(\R^n)=L^p(\R^n)$,
while $L^{p,\infty}(\R^n)$ is the usual weak Lebesgue space. We use the same
notation for Lorentz quasi-norms on the other Euclidean spaces appearing in
the paper. For standard properties of decreasing rearrangements and Lorentz
spaces, see, for example,~\cite[Chapter~1]{grafakos2014classical}.

O'Neil~\cite{oneil1963convolution} proved that, if $1<p_1,p_2,p<\infty$ satisfy Young's exponent relation 
\eqref{Young100} and $0<q_1,q_2,q\leq\infty$ satisfy
\[
\frac1q
\leq
\frac1{q_1}+\frac1{q_2},
\]
then there exists a positive constant
$C=C(p_1,p_2,q_1,q_2,q)$ such that
\[
\|f_1*f_2\|_{p,q}
\leq
C\,
\|f_1\|_{p_1,q_1}
\|f_2\|_{p_2,q_2}
\]
for all $f_1\in L^{p_1,q_1}(\R^n)$ and
$f_2\in L^{p_2,q_2}(\R^n)$. Further developments of convolution theory in Lorentz spaces include the work
of Yap~\cite{yap1969remarks} and
Blozinski~\cite{blozinski1972convolution,blozinski1972proceedings}; the latter
also established the sharpness of the admissible range of indices. Nursultanov
and Tikhonov~\cite{nursultanov2011convolution} studied endpoint and weighted Lorentz estimates, while Nursultanov, Tikhonov, and
Tleukhanova~\cite{nursultanov2018norm} obtained sharper estimates for convolution operator norms in
Lebesgue spaces. These works form part of the linear
background for the multilinear theory developed here.

In this work, we establish  O'Neil-type Lorentz estimates for two classes of translation-invariant multilinear convolution operators. The first is the one-parameter model. Let $\alpha_1,\ldots,\alpha_m$ be nonzero pairwise distinct real numbers. For measurable functions $f_1,\ldots,f_m,g$ on $\R^n$, we define
\begin{equation}\label{eq:one_parameter_operator}
T_g(f_1,\ldots,f_m)(x)
:=
\int_{\R^n}
\Big(\prod_{i=1}^m f_i(x-\alpha_i y)\Big)
g(y)\,\rd y, 
\end{equation}
assuming that the integral converges absolutely for the given $x\in \R^n$. 
Guliyev and Nazirova~\cite{guliyev2008oneil} established a weak-kernel
analogue of Young's inequality for this operator. More precisely, they proved that, if
$1<p_1,\ldots,p_m,r<\infty$ and $1\leq p<\infty$ satisfy
\[
\frac1p
=
\sum_{i=1}^m\frac1{p_i}
+
\frac1r
-
1,
\]
then there exists a positive constant $C=C(n,m,\alpha_1,\ldots,\alpha_m,p_1,\ldots,p_m,r)$
such that
\[
\big\|T_g(f_1,\ldots,f_m)\big\|_p
\leq
C
\Big(\prod_{i=1}^m\|f_i\|_{p_i}\Big)
\|g\|_{r,\infty},
\]
for all $f_i\in L^{p_i}(\R^n)$, $i=1,\ldots,m$, and all $g\in L^{r,\infty}(\R^n)$. Lorentz estimates for weak kernels were subsequently obtained by Guliyev,
Ekincioglu, and Nazirova~\cite{guliyev2015lorentz}. Our first main result complements this theory by capturing the contribution of a finite kernel
Lorentz index to the target space; see Theorem~\ref{thm:one-parameter}.

The second is the full-dimensional model. If $K$ is a measurable function on
$(\R^n)^m$ and $f_1,\ldots,f_m$ are measurable functions on $\R^n$, we define
\begin{equation}\label{eq:full_kernel_operator}
T^K(f_1,\ldots,f_m)(x)
:=
\int_{(\R^n)^m}
\Big(\prod_{i=1}^m f_i(x-y_i)\Big)
K(y_1,\ldots,y_m)
\,\rd y_1\cdots\rd y_m, 
\end{equation}
again under the assumption that the integral converges absolutely. 
Oberlin~\cite{oberlin1988multilinear} established a multilinear version of
Young's convolution inequality for~\eqref{eq:full_kernel_operator}. More
precisely, Oberlin~\cite{oberlin1988multilinear} showed that if $1\leq p_1,\ldots,p_m,r,p\leq\infty$ satisfy
\[
m+\frac1p
=
\sum_{i=1}^m\frac1{p_i}
+
\frac mr,
\]
together with
\[
\frac1p\leq\frac1r
\qquad\text{and}\qquad
1-\frac1r\leq\frac1{p_i},
\quad i=1,\ldots,m,
\]
then there exists a positive constant $C=C(n,m,p_1,\ldots,p_m,r,p)$ such that
\bee\label{Oberlin88}
\big\|T^K(f_1,\ldots,f_m)\big\|_p
\leq
C
\Big(\prod_{i=1}^m\|f_i\|_{p_i}\Big)
\|K\|_r
\eee
for all $f_i\in L^{p_i}(\R^n)$, $i=1,\ldots,m$, and all $K\in L^r((\R^n)^m)$. 

Our second main result gives a Lorentz refinement of Oberlin's estimate \eqref{Oberlin88} in
the interior of this admissible range; see Theorem~\ref{thm:full-kernel}.

As the estimates considered here measure the size of $K$ or $g$, we may as well regard these operators as $(m+1)$-linear mappings in the input functions and the kernel.
Moreover,
\[
\big|T_g(f_1,\ldots,f_m)\big|
\leq
T_{|g|}(|f_1|,\ldots,|f_m|),
\]
and
\[
\big|T^K(f_1,\ldots,f_m)\big|
\leq
T^{|K|}(|f_1|,\ldots,|f_m|).
\]
We therefore use without further mention the standard reduction to nonnegative functions and kernels.
Note that for nonnegative functions the integrals in \eqref{eq:one_parameter_operator} and 
\eqref{eq:full_kernel_operator} are well defined with values in $[0,\infty]$.

Both models fit into a common measure-theoretic framework. If $\mu$ is a 
(nonnegative) regular Borel measure on $(\R^n)^m$, we consider the
translation-invariant $m$-linear operator
\[
T^\mu(f_1,\ldots,f_m)(x)
:=
\int_{(\R^n)^m}
\Big(\prod_{i=1}^m f_i(x-y_i)\Big)
\,\rd\mu(y_1,\ldots,y_m),
\]
defined for nonnegative measurable functions $f_i$. 
Operators of this form have been studied in several settings. The bilinear case with positive kernels was considered in~\cite{grafakos2010positive}, while restricted convolution inequalities and related multilinear operators associated with measures were studied in~\cite{geba2013restricted}.

The full-dimensional operator~\eqref{eq:full_kernel_operator} corresponds to the absolutely continuous measure $\mu_K$ defined by
\[
\rd\mu_K(y_1,\ldots,y_m)
=
K(y_1,\ldots,y_m)\,\rd y_1\cdots\rd y_m.
\]
Thus
\[
T^K=T^{\mu_K}.
\]

For the one-parameter operator~\eqref{eq:one_parameter_operator}, define the
$n$-dimensional subspace
\[
\Gamma
:=
\big\{(\alpha_1y,\ldots,\alpha_my):\, y\in\R^n\big\}
\subseteq
(\R^n)^m,
\]
and let $\mu_g$ be the pushforward of the measure $g(y)\,\rd y$ under the map
\[
y\mapsto(\alpha_1y,\ldots,\alpha_my).
\]
Then $\mu_g$ is supported on $\Gamma$, and
\[
T_g=T^{\mu_g}.
\]

Thus the full-dimensional model corresponds to absolutely continuous kernel
measures, whereas the one-parameter model corresponds to kernel measures
supported on the lower-dimensional subspace $\Gamma$. The different dimensions of the two
kernel domains lead to the different scaling relations appearing in the main
theorems.

More generally, one may consider kernel measures obtained by pushing forward
functions on $\R^{kn}$, where $1\leq k\leq m$. This amounts to choosing linear
maps
\[
L_i:\R^{kn}\to\R^n,
\qquad i=1,\ldots,m,
\]
and considering operators of the form
\[
T_h^L(f_1,\ldots,f_m)(x)
:=
\int_{\R^{kn}}
\prod_{i=1}^m f_i(x-L_i u)\,h(u)\,\rd u,
\]
where $L=(L_1,\ldots,L_m)$. The corresponding kernel measure is supported on
the image of the map
\[
u\mapsto(L_1u,\ldots,L_mu)
\]
from $\R^{kn}$ into $(\R^n)^m$. The full-dimensional operator corresponds to $k=m$, which is seen by writing $u=(y_i)_{i=1}^m$ and taking $L_i u=y_i$. The one-parameter operator
corresponds to $k=1$, with $u=y$ and $L_i u=\alpha_i y$. Unlike the two cases treated here, the intermediate range does not single out a canonical choice of these maps. We therefore leave its systematic study for future work.

\section{Statement of main results}

Guliyev and Nazirova~\cite{guliyev2008oneil} obtained rearrangement inequalities and Lebesgue estimates for multilinear convolutions of the form~\eqref{eq:one_parameter_operator}, while Guliyev, Ekincioglu, and Nazirova~\cite{guliyev2015lorentz} established Lorentz estimates for kernels
in $L^{r,\infty}(\R^n)$. Within the common range of primary exponents, their
result is recovered from the case $s=\infty$ below, together with the standard Lorentz embedding. For finite $s$, our theorem transfers the
additional Lorentz summability of the kernel to the target, yielding a smaller secondary index and hence a stronger target Lorentz space. 
Our first main result concerns the one-parameter model and is as follows:

\begin{theorem}\label{thm:one-parameter}
Let $\alpha_1,\ldots,\alpha_m$ be nonzero pairwise distinct real numbers. Let $1<p_1,\ldots,p_m,r<\infty$ satisfy
\begin{equation}\label{eq:open-banach-condition}
1
<
\sum_{i=1}^m \frac1{p_i}
+
\frac1r
<
2,
\end{equation}
and define $1<p<\infty$ by
\begin{equation}\label{eq:p-definition}
\frac1p=
\sum_{i=1}^m \frac1{p_i}
+
\frac1r
-1.
\end{equation}
Let $0<q_1,\ldots,q_m,s\leq\infty$, and define $0<q\leq\infty$ by
\begin{equation}\label{eq:q-definition}
\frac1q=
\sum_{i=1}^m \frac1{q_i}
+
\frac1s.
\end{equation}
Then there exists a positive constant 
\[C=C(n,m,\alpha_1,\ldots,\alpha_m,p_1,\ldots,p_m,r,q_1,\ldots,q_m,s)\]
 such that
\begin{equation}\label{eq:banach-lorentz}
\big\|T_g(f_1,\ldots,f_m)\big\|_{p,q}
\leq
C 
\Big(\prod_{i=1}^m \|f_i\|_{p_i,q_i}\Big) 
\|g\|_{r,s}
\end{equation}
for all $f_i\in L^{p_i,q_i}(\R^n)$, $i=1,\ldots,m$, and all $g\in L^{r,s}(\R^n)$.
\end{theorem}

The theorem of Guliyev, Ekincioglu, and
Nazirova~\cite{guliyev2015lorentz} also reaches the weak-kernel endpoint $p=1$, under the corresponding condition on the input Lorentz indices, whereas Theorem~\ref{thm:one-parameter} is restricted to $1<p<\infty$. Thus the two results are complementary: their theorem treats this endpoint
for kernels in $L^{r,\infty}(\R^n)$, while
Theorem~\ref{thm:one-parameter} captures the effect of a finite kernel Lorentz index in the open range.

Our second main result concerns the full-dimensional model. It gives a Lorentz refinement of Oberlin's multilinear Young inequality in the interior of its
admissible range.

\begin{theorem}\label{thm:full-kernel}
Let $1<p_1,\ldots,p_m,r,p<\infty$ satisfy
\begin{equation}\label{eq:full-kernel-lorentz-scaling}
m+\frac1p
=
\sum_{i=1}^m\frac1{p_i}
+
\frac mr.
\end{equation}
Assume also that
\begin{equation}\label{eq:full-kernel-lorentz-interior}
\frac1p<\frac1r \qquad \text{and} \qquad
1-\frac1r<\frac1{p_i},
\quad
i=1,\ldots,m.
\end{equation}
Let $0<q_1,\ldots,q_m,s\leq\infty$, and define $0<q\leq\infty$ by
\begin{equation}\label{eq:full-kernel-lorentz-q}
\frac1q
=
\sum_{i=1}^m\frac1{q_i}
+
\frac1s.
\end{equation}
Then there exists a positive constant \[C=C(n,m,p_1,\ldots,p_m,r,q_1,\ldots,q_m,s)\] such that
\begin{equation}\label{eq:full-kernel-lorentz}
\big\|T^K(f_1,\ldots,f_m)\big\|_{p,q}
\leq
C
\Big(\prod_{i=1}^m\|f_i\|_{p_i,q_i}\Big)
\|K\|_{r,s}
\end{equation}
for all $f_i\in L^{p_i,q_i}(\R^n)$, $i=1,\ldots,m$, and all
$K\in L^{r,s}((\R^n)^m)$.
\end{theorem}

We note that, if $q_0$ is defined by
\[
\frac1{q_0}
=
\sum_{i=1}^m\frac1{q_i}
+
\frac1s,
\]
then, by the standard inclusion
\[
L^{p,q_0}(\R^n)\subseteq L^{p,q}(\R^n),
\qquad
q_0\leq q,
\]
the conclusions of both theorems remain valid under the weaker condition
\[
\frac1q
\leq
\sum_{i=1}^m\frac1{q_i}
+
\frac1s.
\]
In this case, the constant may also depend on $q$.

The proofs of Theorems~\ref{thm:one-parameter} and
\ref{thm:full-kernel} are based on an interpolation argument, which proceeds in two steps. First, 
it employs restricted weak-type estimates at the vertices of a local box around the target exponent tuple. For the one-parameter model, these estimates follow from the weak-kernel theorem of Guliyev and Nazirova~\cite[Theorem~4.1]{guliyev2008oneil}; for the full-dimensional model, they follow from Oberlin's multilinear Young inequality~\cite{oberlin1988multilinear}. Viewing the kernel as an additional input, we then apply multilinear Marcinkiewicz interpolation to obtain the stated Lorentz bounds; see Proposition~\ref{prop:lorentz-interpolation}. For
the underlying interpolation theory, see, among others,~\cite{sharpley1977multilinear,grafakos2001remarks,
carro2004restricted,grafakos2012multilinear}.

We next turn to limitations of the full-dimensional theory. In the multilinear case $m\geq2$, at the endpoint $p=r=1$, the strong
$L^1$ estimate for integrable kernels does not admit a weak-$L^1$ counterpart with both the kernel and the target in $L^{1,\infty}$;
see Proposition~\ref{prop:full-kernel-weak-L1-failure}. A more structural obstruction arises from translation invariance itself. H\"ormander proved that a nonzero translation-invariant linear
operator bounded from $L^{p_1}$ to $L^p$ must satisfy $p_1 \leq p$; see~\cite[Theorem~1.1]{hormander1960estimates}. A bilinear analogue appears in the work of Rodr\'iguez-L\'opez~\cite{rodriguezlopez2013homomorphism},
and the corresponding multilinear statement in Lebesgue spaces is also
known; see~\cite{grafakos2014modern}.

Our third main result extends this obstruction to the multilinear Lorentz setting for nonnegative full-dimensional kernels. In the following result, we recover the necessary condition on the primary exponents and, in the critical case, obtain an additional necessary condition
on the secondary Lorentz indices.

\begin{theorem}\label{thm:lorentz-zero-kernel}
Let $0<p_1,\ldots,p_m,p<\infty$
and $0<q_1,\ldots,q_m,q\leq\infty$.
Let $K$ be a nonnegative measurable function on $(\R^n)^m$ that is nonzero on a set of positive measure. If there exists a positive constant $C_K$ such that
\begin{equation}\label{eq:lorentz-zero-kernel-boundedness}
\big\|
T^K(f_1,\ldots,f_m)
\big\|_{p,q}
\leq
C_K
\prod_{i=1}^m
\|f_i\|_{p_i,q_i}
\end{equation}
for all $f_i\in L^{p_i,q_i}(\R^n)$, $i=1,\ldots,m$, then we must have
\begin{equation}\label{eq:primary-zero-kernel-condition}
\frac1p
\leq
\sum_{i=1}^m\frac1{p_i}.
\end{equation}
Moreover, if equality holds in~\eqref{eq:primary-zero-kernel-condition}, then necessarily
\begin{equation}\label{eq:secondary-zero-kernel-condition}
\frac1q
\leq
\sum_{i=1}^m\frac1{q_i}.
\end{equation}
\end{theorem}

Equivalently, for a nonnegative measurable kernel $K$, boundedness
as in~\eqref{eq:lorentz-zero-kernel-boundedness} forces $K=0$ almost
everywhere whenever either
\[
\sum_{i=1}^m\frac1{p_i}<\frac1p,
\]
or
\[
\sum_{i=1}^m\frac1{p_i}=\frac1p
\qquad\text{and}\qquad
\sum_{i=1}^m\frac1{q_i}<\frac1q.
\]

Although conclusions of the type \eqref{eq:primary-zero-kernel-condition}   were first obtained by H\"ormander~\cite{hormander1960estimates} for non-necessarily positive kernels, in this paper we focus on  positive operators and for this reason we work with   positive kernels. Our extension of H\"ormander's theorem to Lorentz spaces   presented here is 
based on an approach that is different from those in 
~\cite{hormander1960estimates} and in ~\cite{rodriguezlopez2013homomorphism}.

We now return to the one-parameter model, where
Theorem~\ref{thm:one-parameter} concerns target exponents $p>1$. Our fourth main result shows that a natural scale-invariant family of restricted weak-type estimates fails throughout the range $p<1$, uniformly over measurable kernel sets and already in the one-dimensional bilinear case.
The proof uses an additive-combinatorial construction in which an imbalance
between sumset and difference-set cardinalities is amplified by passing to Cartesian powers and then transferred to measurable sets.

To state the result, we specialize~\eqref{eq:one_parameter_operator} to
$n=1$, $m=2$, $\alpha_1=-1$, and $\alpha_2=1$, and retain the notation
$T_g$:
\[
T_g(f_1,f_2)(x)
=
\int_{\R}
f_1(x+y) \, f_2(x-y) \, g(y)\,\rd y.
\]

\begin{theorem}\label{thm:counterexamples}
Let $\beta > 0$. There is no constant $C>0$ such that
\begin{equation}\label{eq:counterexamples}
\big\|
T_{\chi_A}(\chi_{E_1},\chi_{E_2})
\big\|_{1/(1+\beta),\infty}
\leq
C \, |E_1|\,|E_2|\,|A|^\beta
\end{equation}
for all measurable sets $E_1,E_2,A\subseteq\R$ of finite measure. 
\end{theorem}

Thus, for every $\beta>0$, the associated trilinear mapping is not of
restricted weak type
\[
\left(1,1,\frac1\beta;\frac1{1+\beta}\right).
\]
In particular, taking $\beta=1$ shows that
\[
\big\|
T_{\chi_A}(\chi_{E_1},\chi_{E_2})
\big\|_{1/2,\infty}
\leq
C \, |E_1|\,|E_2|\,|A|
\]
cannot hold with a constant independent of the measurable sets
$E_1,E_2,A\subseteq\R$ of finite measure. Related questions for
one-parameter trilinear operators were studied by
Christ~\cite{christ2001trilinear}.

This uniform failure over arbitrary kernel sets does not preclude estimates for fixed kernels or for classes with additional structure. In the final part of the paper, we prove an $L^1\times L^1\to L^{1/2}$ estimate for cube kernels and obtain estimates, including some with target exponent $p < 1$, for
radially decreasing kernels.

\smallskip
\subsection*{Outline}

Section~\ref{sec:restricted-weak} establishes the restricted weak-type
estimates used in the proofs of the two Lorentz bounds. 
The section concludes with the multilinear
interpolation proposition that converts a local family of such estimates into
a Lorentz bound.

Sections~\ref{sec:one-parameter} and~\ref{sec:full-dimensional} prove
Theorems~\ref{thm:one-parameter} and~\ref{thm:full-kernel}, respectively. In
each case, we construct a local box of admissible exponents and apply the
interpolation result from Section~\ref{sec:restricted-weak}. 

Section~\ref{sec:sharpness-full} concerns limitations of the
full-dimensional theory. It proves the failure of the weak $L^1$ endpoint
and establishes the H\"ormander-type necessary conditions in
Theorem~\ref{thm:lorentz-zero-kernel}. Section~\ref{sec:counterexamples}
proves Theorem~\ref{thm:counterexamples} by means of an
additive-combinatorial construction. The final part of that section gives
complementary positive results for kernels with additional structure,
including a strong endpoint estimate for cube kernels and bounds with target
exponent below $1$ for radially decreasing kernels.

\section{Restricted weak-type estimates and multilinear interpolation} \label{sec:restricted-weak}

In this section we establish the restricted weak-type estimates used in the proofs of Theorem~\ref{thm:one-parameter} and Theorem~\ref{thm:full-kernel}. 


\subsection{One-parameter estimates}

The restricted weak-type estimates needed for the one-parameter model are consequences of the weak-kernel estimate of Guliyev and Nazirova~\cite[Theorem~4.1]{guliyev2008oneil}.

\begin{lemma}\label{lem:one-parameter-restricted-weak}
Let $\alpha_1,\ldots,\alpha_m$ be nonzero pairwise distinct real numbers. Let
$0<\gamma_1,\ldots,\gamma_m,\beta<1$ satisfy
\[
1<
\sum_{i=1}^m\gamma_i+\beta
\leq
2.
\]
Define $1\leq p_0<\infty$ by
\[
\frac1{p_0}
=
\sum_{i=1}^m\gamma_i+\beta-1.
\]
Then there exists a positive constant $C=C(n,m,\alpha_1,\ldots,\alpha_m,
\gamma_1,\ldots,\gamma_m,\beta)$
such that
\begin{equation}\label{eq:banach-vertex-estimate}
\big\|T_{\chi_A}(\chi_{E_1},\ldots,\chi_{E_m})\big\|_{p_0,\infty}
\leq
C
\Big(\prod_{i=1}^m |E_i|^{\gamma_i}\Big)
|A|^\beta
\end{equation}
for all measurable sets $E_1,\ldots,E_m,A\subseteq\R^n$ of finite measure.
\end{lemma}

\begin{proof}
Set
\[
p_i:=\frac1{\gamma_i},
\quad
i=1,\ldots,m,
\qquad \text{and} \qquad
r:=\frac1\beta.
\]
Since $0<\gamma_1,\ldots,\gamma_m,\beta<1$, we have
\[
1<p_1,\ldots,p_m,r<\infty.
\]
Moreover, by the definition of $p_0$,
\[
\frac1{p_0}
=
\sum_{i=1}^m\gamma_i+\beta-1
=
\sum_{i=1}^m\frac1{p_i}
+
\frac1r
-
1.
\]
Since $1\leq p_0<\infty$, the weak-kernel estimate of Guliyev and Nazirova~\cite[Theorem~4.1]{guliyev2008oneil} gives
\[
\big\|T_{\chi_A}(\chi_{E_1},\ldots,\chi_{E_m})\big\|_{p_0}
\leq
C
\Big(\prod_{i=1}^m
\|\chi_{E_i}\|_{p_i}\Big)
\|\chi_A\|_{r,\infty},
\]
for all measurable sets $E_1,\ldots,E_m,A\subseteq\R^n$ of finite measure. This implies~\eqref{eq:banach-vertex-estimate}.
\end{proof}

\subsection{Full-dimensional estimates}

The corresponding restricted weak-type estimates for the full-dimensional model are based on Oberlin's multilinear Young inequality~\cite{oberlin1988multilinear}.

\begin{lemma}\label{lem:full-kernel-restricted-weak}
Let $0<\gamma_1,\ldots,\gamma_m,\beta<1$ satisfy
\[
0<
\sum_{i=1}^m\gamma_i
+
m\beta
-
m
<1.
\]
Define $1<p_0<\infty$ by
\[
\frac1{p_0}
=
\sum_{i=1}^m\gamma_i
+
m\beta
-
m.
\]
Assume also that
\[
\frac1{p_0}< \beta, \qquad \text{and} \qquad
1-\beta< \gamma_i,
\quad
i=1,\ldots,m.
\]
Then there exists a positive constant $C=C(n,m,\gamma_1,\ldots,\gamma_m,\beta)$
such that
\begin{equation}\label{eq:full-kernel-restricted-weak}
\big\|T^{\chi_\Omega}(\chi_{E_1},\ldots,\chi_{E_m})\big\|_{p_0,\infty}
\leq
C
\Big(\prod_{i=1}^m |E_i|^{\gamma_i}\Big)
|\Omega|^\beta
\end{equation}
for all measurable sets $E_1,\ldots,E_m\subseteq\R^n$ and
$\Omega \subseteq(\R^n)^m$ of finite measure.
\end{lemma}

\begin{proof}
Set
\[
 p_i:=\frac1{\gamma_i},
\quad i=1,\ldots,m, 
\qquad \text{and} \qquad 
 r:=\frac1\beta.
\]
Since $0<\gamma_1,\ldots,\gamma_m,\beta<1$, we have \[1< p_1,\ldots, p_m, r<\infty.
\]
Moreover,
\[
m+\frac1{p_0}
=
\sum_{i=1}^m \gamma_i
+
m\beta
=
\sum_{i=1}^m \frac1{p_i}
+
\frac m{r}.
\]
The remaining assumptions give
\[
\frac1{p_0}<\beta=\frac1{r},
\]
and
\[
1-\frac1{r}
=
1-\beta
<
\gamma_i
=
\frac1{p_i},
\qquad
i=1,\ldots,m.
\]
Thus Oberlin's multilinear Young inequality~\cite{oberlin1988multilinear} applies and gives
\[
\big\|T^{\chi_\Omega}(\chi_{E_1},\ldots,\chi_{E_m})\big\|_{p_0}
\leq
C
\Big(\prod_{i=1}^m
\|\chi_{E_i}\|_{p_i}
\Big)
\|\chi_\Omega\|_{r},
\]
for all measurable sets $E_1,\ldots,E_m\subseteq\R^n$ and
$\Omega \subseteq(\R^n)^m$ of finite measure, which implies~\eqref{eq:full-kernel-restricted-weak}.
\end{proof}

\subsection{Multilinear interpolation}
\label{sec:lorentz-lifting}

The following proposition is a consequence of multilinear Marcinkiewicz interpolation; see~\cite{sharpley1977multilinear,grafakos2001remarks,
grafakos2012multilinear, grafakos2014modern}. The parameter $\kappa$ specifies the dimension $\kappa n$ of the kernel domain, with $\kappa=1$ for the one-parameter model and $\kappa=m$ for the full-dimensional model.

\begin{proposition}\label{prop:lorentz-interpolation}
Fix one of the following two cases: $\kappa=1$ and
$\mathcal T_h=T_h$, or $\kappa=m$ and $\mathcal T_h=T^h$. Let $0<a_1,\ldots,a_m,b<1$ satisfy
\[
0<
\sum_{i=1}^m a_i+\kappa b-\kappa
<
1,
\]
and define $1<p<\infty$ by
\[
\frac1p
=
\sum_{i=1}^m a_i+\kappa b-\kappa.
\]
Assume that there exist numbers $a_i^-$, $a_i^+$ with 
\[
0<a_i^-<a_i<a_i^+<1, \qquad i=1,\ldots,m,
\]
and $b_-$, $b_+$ with 
\[
0<b_-<b<b_+<1
\]
such that, for every sign vector
\[
\eps=(\eps_1,\ldots,\eps_m,\eps_0)\in\{-,+\}^{m+1},
\]
the number $p_\eps$ defined by
\[
\frac1{p_\eps}
=
\sum_{i=1}^m a_i^{\eps_i}
+
\kappa b_{\eps_0}
-
\kappa
\]
satisfies $1<p_\eps<\infty$, and the restricted weak-type estimate
\begin{equation}\label{eq:abstract-restricted-weak}
\big\|
\mathcal T_{\chi_A}
(\chi_{E_1},\ldots,\chi_{E_m})
\big\|_{p_\eps,\infty}
\leq
C_0
\Big(
\prod_{i=1}^m |E_i|^{a_i^{\eps_i}}
\Big)
|A|^{b_{\eps_0}}
\end{equation}
holds for all measurable sets $E_1,\ldots,E_m\subseteq\R^n$ and $A\subseteq(\R^n)^\kappa$
of finite measure, where $C_0$ is independent of $\eps$ and of the sets. Let $0<q_1,\ldots,q_m,s\leq\infty$ and define $0<q\leq\infty$ by
\[
\frac1q
=
\sum_{i=1}^m\frac1{q_i}
+
\frac1s.
\]
Then there exists a positive constant $C$ such that
\begin{equation}\label{eq:abstract-lorentz-lifting}
\big\|
\mathcal T_h(f_1,\ldots,f_m)
\big\|_{p,q}
\leq
C
\Big(
\prod_{i=1}^m
\|f_i\|_{1/a_i,q_i}
\Big)
\|h\|_{1/b,s}
\end{equation}
for all $f_i\in L^{1/a_i,q_i}(\R^n)$, $i=1,\ldots,m$, and all $h\in L^{1/b,s}\big((\R^n)^\kappa\big)$.
\end{proposition}

\begin{proof}
We verify that the exponent tuple in the conclusion lies in the interior of the convex hull of the tuples occurring in the restricted
weak-type estimates. 

Set 
\[
\theta_i
:=
\frac{a_i-a_i^-}{a_i^+-a_i^-}, \quad i=1,\ldots,m, \qquad \text{and} \qquad \theta_0
:=
\frac{b-b_-}{b_+-b_-}.
\]
Since
\[
a_i^-<a_i<a_i^+,
\qquad i=1,\ldots,m,
\]
and
\[
b_-<b<b_+,
\]
we have
\[
0<\theta_i<1,
\qquad i=0,1,\ldots,m.
\]
Moreover,
\[
a_i
=
(1-\theta_i)a_i^-+\theta_i a_i^+,
\qquad
i=1,\ldots,m,
\]
and
\[
b
=
(1-\theta_0)b_-+\theta_0b_+.
\]

For each $i=0,\ldots,m$, define
\[
w_i(-):=1-\theta_i,
\qquad
w_i(+):=\theta_i.
\]
For
\[
\eps=(\eps_1,\ldots,\eps_m,\eps_0)
\in\{-,+\}^{m+1},
\]
set
\[
\lambda_\eps
:=
\prod_{i=0}^m w_i(\eps_i) > 0.
\]
We have
\begin{align*}
\sum_{\eps\in\{-,+\}^{m+1}}\lambda_\eps
&=
\prod_{i=0}^m
\Bigg(
\sum_{\sigma\in\{-,+\}}w_i(\sigma)
\Bigg)\\
&=
\prod_{i=0}^m
\big((1-\theta_i)+\theta_i\big)\\
&=
1.
\end{align*}
Moreover, for each $i=1,\ldots,m$,
\begin{align*}
\sum_{\eps\in\{-,+\}^{m+1}}
\lambda_\eps \, a_i^{\eps_i}
&=
\Bigg(
\sum_{\sigma\in\{-,+\}}
w_i(\sigma) \, a_i^\sigma
\Bigg)
\prod_{\substack{0\leq j\leq m\\ j\neq i}}
\Bigg(
\sum_{\sigma\in\{-,+\}}w_j(\sigma)
\Bigg)\\
&=
(1-\theta_i)a_i^-+\theta_i a_i^+\\
&=
a_i.
\end{align*}
Similarly,
\begin{align*}
\sum_{\eps\in\{-,+\}^{m+1}}
\lambda_\eps \, b_{\eps_0}
&=
\Bigg(
\sum_{\sigma\in\{-,+\}}
w_0(\sigma) \, b_\sigma
\Bigg)
\prod_{j=1}^m
\Bigg(
\sum_{\sigma\in\{-,+\}}w_j(\sigma)
\Bigg)\\
&=
(1-\theta_0)b_-+\theta_0b_+\\
&=
b.
\end{align*}

Now, for every $\eps\in\{-,+\}^{m+1}$ we define
\[
\mathbf v_\eps
:=
\left(
a_1^{\eps_1},
\ldots,
a_m^{\eps_m},
b_{\eps_0},
\frac1{p_\eps}
\right),
\]
and 
\[
\mathbf v
:=
\left(
a_1,\ldots,a_m,b,\frac1p
\right).
\]
Since
\[
\frac1{p_\eps}
=
\sum_{i=1}^m a_i^{\eps_i}
+
\kappa b_{\eps_0}
-
\kappa,
\]
we obtain
\begin{align*}
\sum_{\eps\in\{-,+\}^{m+1}}
\lambda_\eps \, \frac1{p_\eps}
&=
\sum_{\eps \in\{-,+\}^{m+1}}
\lambda_\eps
\Bigg(
\sum_{i=1}^m a_i^{\eps_i}
+
\kappa b_{\eps_0}
-
\kappa
\Bigg)\\
&=
\sum_{i=1}^m a_i+\kappa b-\kappa\\
&=
\frac1p.
\end{align*}
Consequently,
\[
\mathbf v
=
\sum_{\eps\in\{-,+\}^{m+1}}
\lambda_\eps\mathbf v_\eps.
\]
This proves that $\mathbf v$ is a convex combination of the vectors $\mathbf v_\eps$, $\eps\in\{-,+\}^{m+1}$. The result follows.
\end{proof}

\section{Lorentz estimates for the one-parameter model}\label{sec:one-parameter}

In this section we prove Theorem~\ref{thm:one-parameter}. We first verify the local-box condition required by Proposition~\ref{prop:lorentz-interpolation} with $\kappa=1$, and then combine the resulting local box with the restricted weak-type estimates established in
Section~\ref{sec:restricted-weak}.

\subsection{A local box of exponents}

\begin{lemma}\label{lem:one-parameter-local-box}
Let $0<a_1,\ldots,a_m,b<1$ and suppose that
\[
1<\sum_{i=1}^m a_i+b<2.
\]
Then there exist numbers $a_i^-$ and $a_i^+$ with
\[
0<a_i^-<a_i<a_i^+<1,
\qquad i=1,\ldots,m,
\]
and also numbers $b_-, b_+$ with 
\[
0<b_-<b<b_+<1,
\]
such that, for every sign vector
$\eps=(\eps_1,\ldots,\eps_m,\eps_0)\in\{-,+\}^{m+1}$, one has
\begin{equation}\label{eq:banach-cube-conditions}
1
<
\sum_{i=1}^m a_i^{\eps_i}
+
b_{\eps_0}
<
2.
\end{equation}
\end{lemma}

\begin{proof}
Set
\[
\Sigma:=\sum_{i=1}^m a_i+b.
\]
By assumption, $1<\Sigma<2$. Choose $\delta>0$ so small that
\[
\delta
<
\frac12
\min\left\{
a_1,\ldots,a_m,b,
1-a_1,\ldots,1-a_m,1-b,
\frac{\Sigma-1}{m+1},
\frac{2-\Sigma}{m+1}
\right\}.
\]
Then set
\[
a_i^-=a_i-\delta,
\qquad
a_i^+=a_i+\delta,
\qquad i=1,\ldots,m,
\]
and
\[
b_-=b-\delta,
\qquad
b_+=b+\delta.
\]
The inequalities in the statement follow from the choice of $\delta$. Moreover, for every
$\eps\in\{-,+\}^{m+1}$,
\[
\Sigma-(m+1)\delta
\leq
\sum_{i=1}^m a_i^{\eps_i}
+
b_{\eps_0}
\leq
\Sigma+(m+1)\delta.
\]
Since $\Sigma-(m+1)\delta>1$ and $\Sigma+(m+1)\delta<2$, this proves~\eqref{eq:banach-cube-conditions}. \end{proof}

\subsection{Proof of Theorem~\ref{thm:one-parameter}}

We apply Proposition~\ref{prop:lorentz-interpolation} with
\[
\mathcal T_h=T_h
\qquad\text{and}\qquad
\kappa=1.
\]

Set
\[
a_i:=\frac1{p_i},
\qquad
i=1,\ldots,m,
\qquad
b:=\frac1r,
\]
and take $\kappa=1$. By~\eqref{eq:p-definition},
\[
\frac1p
=
\sum_{i=1}^m a_i+b-1,
\]
while~\eqref{eq:open-banach-condition} gives
\[
1<
\sum_{i=1}^m a_i+b
<
2.
\]
Hence Lemma~\ref{lem:one-parameter-local-box} provides numbers $a_i^-$, $a_i^+$ with 
\[
0<a_i^-<a_i<a_i^+<1,
\qquad
i=1,\ldots,m,
\]
and numbers $b_-$, $b_+$ with 
\[
0<b_-<b<b_+<1,
\]
such that, for every
\[
\eps=(\eps_1,\ldots,\eps_m,\eps_0)
\in\{-,+\}^{m+1},
\]
the number $p_\eps$ defined by
\[
\frac1{p_\eps}
=
\sum_{i=1}^m a_i^{\eps_i}
+
b_{\eps_0}
-
1
\]
satisfies
\[
1<p_\eps<\infty.
\]

Applying Lemma~\ref{lem:one-parameter-restricted-weak} with
\[
\gamma_i=a_i^{\eps_i},
\qquad
i=1,\ldots,m,
\qquad
\beta=b_{\eps_0},
\]
we obtain
\[
\big\|
T_{\chi_A}(\chi_{E_1},\ldots,\chi_{E_m})
\big\|_{p_\eps,\infty}
\leq
C_\eps
\Big(\prod_{i=1}^m |E_i|^{a_i^{\eps_i}}\Big)
|A|^{b_{\eps_0}}
\]
for all measurable sets $E_1,\ldots,E_m,A\subseteq\R^n$ of finite measure.
Since there are only finitely many sign vectors $\eps$, the constants $C_\eps$ may be replaced by a single constant independent of~$\eps$.

All the hypotheses of Proposition~\ref{prop:lorentz-interpolation} are therefore
satisfied with $\kappa=1$. Using also~\eqref{eq:q-definition}, we conclude that
\[
\big\|T_g(f_1,\ldots,f_m)\big\|_{p,q}
\leq
C
\Big(\prod_{i=1}^m\|f_i\|_{1/a_i,q_i}\Big)
\|g\|_{1/b,s}.
\]
Since
\[
\frac1{a_i}=p_i,
\quad
i=1,\ldots,m,
\qquad \text{and} \qquad
\frac1b=r,
\]
this is precisely~\eqref{eq:banach-lorentz}.
\qed

\section{Lorentz estimates for the full-dimensional model}
\label{sec:full-dimensional}

In this section we prove Theorem~\ref{thm:full-kernel}. We first verify the local-box condition required by Proposition~\ref{prop:lorentz-interpolation} with $\kappa=m$, and then combine the resulting local box with the restricted weak-type estimates established in Section~\ref{sec:restricted-weak}.

\subsection{A local box of exponents}

\begin{lemma}\label{lem:full-dimensional-local-box}
Let $0<a_1,\ldots,a_m,b<1$ and suppose that
\[
0<
\sum_{i=1}^m a_i+mb-m
<1, \qquad 
\sum_{i=1}^m a_i+mb-m<b,
\]
and
\[
1-b<a_i,
\qquad i=1,\ldots,m.
\]
Then there exist numbers  $a_i^-$, $a_i^+$ with 
\[
0<a_i^-<a_i<a_i^+<1,
\qquad i=1,\ldots,m,
\]
and numbers  $b_-$, $b_+$ with 
\[
0<b_-<b<b_+<1,
\]
such that, for every sign vector
\[
\eps=(\eps_1,\ldots,\eps_m,\eps_0)\in\{-,+\}^{m+1},
\]
the quantity
\[
\rho_\eps
:=
\sum_{i=1}^m a_i^{\eps_i}
+
m b_{\eps_0}
-
m
\]
satisfies
\[
0<\rho_\eps<1, \qquad \rho_\eps<b_{\eps_0}, \qquad 
\text{and} \qquad 
1-b_{\eps_0}<a_i^{\eps_i},
\quad i=1,\ldots,m.
\]
\end{lemma}

\begin{proof}
Set
\[
\rho:=
\sum_{i=1}^m a_i+mb-m.
\]
By assumption,
\[
0<\rho<1,
\qquad
\rho<b,
\qquad
a_i+b-1>0,
\quad i=1,\ldots,m.
\]
Choose $\delta>0$ so small that
\[
\delta
<
\min\left\{
\min_{1\leq i\leq m}a_i,
\min_{1\leq i\leq m}(1-a_i),
b,
1-b
\right\},
\]
\[
2m\delta<\min\{\rho,1-\rho\},
\]
\[
(2m-1)\delta<b-\rho,
\]
and
\[
2\delta<
\min_{1\leq i\leq m}(a_i+b-1).
\]
Define
\[
a_i^-:=a_i-\delta,
\qquad
a_i^+:=a_i+\delta,
\qquad i=1,\ldots,m,
\]
and
\[
b_-:=b-\delta,
\qquad
b_+:=b+\delta.
\]
The first condition on $\delta$ gives
\[
0<a_i^-<a_i<a_i^+<1,
\qquad i=1,\ldots,m,
\]
and
\[
0<b_-<b<b_+<1.
\]

Now fix
\[
\eps=(\eps_1,\ldots,\eps_m,\eps_0)\in\{-,+\}^{m+1}.
\]
Since each $a_i^{\eps_i}$ differs from $a_i$ by at most $\delta$, and
$b_{\eps_0}$ differs from $b$ by at most $\delta$, we have
\[
\left|
\rho_\eps-\rho
\right|
=
\left|
\sum_{i=1}^m (a_i^{\eps_i}-a_i)
+
m(b_{\eps_0}-b)
\right|
\leq
2m\delta.
\]
Using the second condition on $\delta$ we obtain
\[
\rho_\eps
\geq
\rho-2m\delta
>
0
\]
and
\[
\rho_\eps
\leq
\rho+2m\delta
<
1.
\]
Thus
\[
0<\rho_\eps<1.
\]

Next we prove that $\rho_\eps<b_{\eps_0}$. Note that
\[
b-\rho
=
b-
\left(
\sum_{i=1}^m a_i+mb-m
\right)
>0.
\]
Moreover,
\begin{align*}
\left|
(b_{\eps_0}-\rho_\eps)-(b-\rho)
\right| & =
\left|
-\sum_{i=1}^m(a_i^{\eps_i}-a_i)
-(m-1)(b_{\eps_0}-b)
\right|
\\
&\leq
m\delta+(m-1)\delta
=
(2m-1)\delta.
\end{align*}
Using the third condition on $\delta$ we get
\[
b_{\eps_0}-\rho_\eps
\geq
b-\rho-(2m-1)\delta
>
0.
\]
Hence
\[
\rho_\eps<b_{\eps_0}.
\]

Finally, for each $i=1,\ldots,m$, we have
\[
a_i^{\eps_i}+b_{\eps_0}-1
=
a_i+b-1
+
(a_i^{\eps_i}-a_i)
+
(b_{\eps_0}-b).
\]
Therefore
\[
a_i^{\eps_i}+b_{\eps_0}-1
\geq
a_i+b-1-2\delta.
\]
Using the fourth condition on $\delta$ we obtain
\[
a_i^{\eps_i}+b_{\eps_0}-1>0,
\]
or equivalently,
\[
1-b_{\eps_0}<a_i^{\eps_i},
\qquad
i=1,\ldots,m.
\]
This proves all the required inequalities.
\end{proof}

\subsection{Proof of Theorem~\ref{thm:full-kernel}}

We apply Proposition~\ref{prop:lorentz-interpolation} with
\[
\mathcal T_h=T^h
\qquad\text{and}\qquad
\kappa=m.
\]

Set
\[
a_i:=\frac1{p_i},
\qquad
i=1,\ldots,m,
\qquad
b:=\frac1r,
\]
and take $\kappa=m$. By~\eqref{eq:full-kernel-lorentz-scaling},
\[
\frac1p
=
\sum_{i=1}^m a_i+mb-m.
\]
Since $1<p<\infty$, we have
\[
0<
\sum_{i=1}^m a_i+mb-m
<
1.
\]
Moreover, the assumptions~\eqref{eq:full-kernel-lorentz-interior} give
\[
\sum_{i=1}^m a_i+mb-m<b
\]
and
\[
1-b<a_i,
\qquad
i=1,\ldots,m.
\]
Hence Lemma~\ref{lem:full-dimensional-local-box} gives numbers $a_i^-$, $a_i^+$ with 
\[
0<a_i^-<a_i<a_i^+<1,
\qquad
i=1,\ldots,m,
\]
and $b_-$, $b_+$ with 
\[
0<b_-<b<b_+<1,
\]
such that, for every
\[
\eps=(\eps_1,\ldots,\eps_m,\eps_0)
\in\{-,+\}^{m+1},
\]
the quantity
\[
\rho_\eps
:=
\sum_{i=1}^m a_i^{\eps_i}
+
m b_{\eps_0}
-
m
\]
satisfies
\[
0<\rho_\eps<1,
\qquad
\rho_\eps<b_{\eps_0},
\]
and
\[
1-b_{\eps_0}<a_i^{\eps_i},
\qquad
i=1,\ldots,m.
\]
Define $p_\eps$ by
\[
\frac1{p_\eps}:=\rho_\eps.
\]
Then
\[
1<p_\eps<\infty.
\]

Applying Lemma~\ref{lem:full-kernel-restricted-weak} with
\[
\gamma_i=a_i^{\eps_i},
\qquad
i=1,\ldots,m,
\qquad
\beta=b_{\eps_0},
\]
we obtain
\[
\big\|
T^{\chi_\Omega}
(\chi_{E_1},\ldots,\chi_{E_m})
\big\|_{p_\eps,\infty}
\leq
C_\eps
\Big(\prod_{i=1}^m |E_i|^{a_i^{\eps_i}}\Big)
|\Omega|^{b_{\eps_0}}
\]
for all measurable sets $E_1,\ldots,E_m\subseteq\R^n$ and
$\Omega\subseteq(\R^n)^m$ of finite measure. Since there are only finitely many sign vectors, the
constants $C_\eps$ may be replaced by a single constant independent of
$\eps$.

All the hypotheses of Proposition~\ref{prop:lorentz-interpolation} are therefore
satisfied with $\kappa=m$. Using also~\eqref{eq:full-kernel-lorentz-q}, we conclude that
\[
\big\|
T^K(f_1,\ldots,f_m)
\big\|_{p,q}
\leq
C
\Big(\prod_{i=1}^m
\|f_i\|_{1/a_i,q_i}
\Big)
\|K\|_{1/b,s}.
\]
Since
\[
\frac1{a_i}=p_i,
\qquad
i=1,\ldots,m,
\qquad
\frac1b=r,
\]
this is precisely~\eqref{eq:full-kernel-lorentz}.
\qed

\section{Sharpness results for the full-dimensional model}
\label{sec:sharpness-full}

We complement the estimates of Section~\ref{sec:full-dimensional} with two results describing limitations of the full-dimensional theory. We first show that the strong endpoint estimate for integrable kernels does not extend to
kernels in weak $L^1$. We then derive necessary conditions for the boundedness of $T^K$ when
the kernel is nonnegative and nonzero, including an additional restriction on
the Lorentz indices in the case
\[
\frac1p=\sum_{i=1}^m\frac1{p_i}.
\]
\subsection{Failure at the weak \texorpdfstring{\boldmath $L^1$}{L1} endpoint}

If $K\in L^1((\R^n)^m)$, $1\leq p_1,\ldots,p_m\leq\infty$, and
\[
\sum_{i=1}^m\frac1{p_i}=1,
\]
then Tonelli's theorem and H\"older's inequality give
\[
\big\|T^K(f_1,\ldots,f_m)\big\|_1
\leq
\Big(\prod_{i=1}^m\|f_i\|_{p_i}\Big)\|K\|_1.
\]
The next proposition shows that no uniform estimate of the same form holds when
the $L^1$ norm of the kernel is replaced by its weak $L^1$ quasi-norm.

\begin{proposition}\label{prop:full-kernel-weak-L1-failure}
Let $m\geq2$, and let $1<p_1,\ldots,p_m<\infty$ satisfy
\[
\sum_{i=1}^m\frac1{p_i}=1.
\]
There is no constant $C>0$ such that
\begin{equation}\label{eq:full-kernel-weak-L1-failure}
\big\|T^K(f_1,\ldots,f_m)\big\|_{1,\infty}
\leq
C
\Big(\prod_{i=1}^m\|f_i\|_{p_i}\Big)
\|K\|_{1,\infty}
\end{equation}
for all $f_i\in L^{p_i}(\R^n)$, $i=1,\ldots,m$, and all
$K\in L^{1,\infty}((\R^n)^m)$.
\end{proposition}

\begin{proof}
Let
\[
D:=mn,
\]
and let $\omega_D$ denote the measure of the unit ball in $\R^D$. For $R>16$, define
\[
K_R(Y):=|Y|^{-D}\chi_{\{1<|Y|<R/8\}}(Y), \qquad Y \in \R^D.
\]

We first check that the weak $L^1$ norm of $K_R$ is bounded independently of
$R$. If $\lambda\geq1$, then $d_{K_R}(\lambda)=0$, since $K_R(Y)\leq1$ on its support.
If $0<\lambda<1$, then
\[
\{K_R>\lambda\}
\subseteq
\{|Y|^{-D}>\lambda\}
=
\{|Y|<\lambda^{-1/D}\}.
\]
Therefore
\[
\big|\{K_R > \lambda \}\big|
\leq
\big|\{|Y|<\lambda^{-1/D}\}\big|
=
\omega_D \lambda^{-1}.
\]
Thus, for every $\lambda>0$,
\[
\lambda \, \big|\{K_R > \lambda \}\big|\leq \omega_D.
\]
Hence
\begin{equation}\label{eq:KR-weak-L1}
\|K_R\|_{1,\infty}
=
\sup_{\lambda>0}\lambda \, \big|\{K_R > \lambda \}\big|
\leq
\omega_D,
\end{equation}
uniformly in $R$.

Now set
\[
Q_R:=[0,R]^n,
\qquad
f_i:=\chi_{Q_R},
\qquad i=1,\ldots,m.
\]
Since
\[
\sum_{i=1}^m\frac1{p_i}=1,
\]
we have
\begin{equation}\label{eq:input-product-full-kernel-counter}
\prod_{i=1}^m \|f_i\|_{p_i}
=
\prod_{i=1}^m |Q_R|^{1/p_i}
=
\prod_{i=1}^m R^{n/p_i}
=
R^n.
\end{equation}

Let
\[
Q_R':=\left[\frac R4,\frac{3R}{4}\right]^n.
\]
If $x\in Q_R'$ and $Y=(y_1,\ldots,y_m)$ satisfies $1<|Y|<R/8$, then
\[
|y_i|\leq |Y|<\frac R8,
\qquad i=1,\ldots,m.
\]
In particular, each coordinate of $y_i$ has absolute value less than $R/8$.
Since each coordinate of $x$ lies between $R/4$ and $3R/4$, it follows that
\[
x-y_i\in [0,R]^n=Q_R,
\qquad i=1,\ldots,m.
\]
Therefore
\[
\prod_{i=1}^m f_i(x-y_i)=1
\]
whenever $x\in Q_R'$ and $1<|Y|<R/8$. Consequently,
\[
T^{K_R}(f_1,\ldots,f_m)(x)
\geq
\int_{\{1<|Y|<R/8\}} |Y|^{-D}\,\rd Y
\]
for every $x\in Q_R'$.

Using polar coordinates in $\R^D$, we get
\begin{align*}
\int_{\{1<|Y|<R/8\}} |Y|^{-D}\,\rd Y
& =
D\omega_D
\int_1^{R/8}
\rho^{-D}\rho^{D-1}\,\rd \rho \\
& =
D\omega_D \log(R/8).
\end{align*}
Thus
\[
T^{K_R}(f_1,\ldots,f_m)(x)
\geq
D\omega_D\log(R/8)
\]
for every $x\in Q_R'$.

Taking
\[
\lambda_R:=\frac12D\omega_D\log(R/8),
\]
we have
\[
Q_R'\subseteq
\big\{T^{K_R}(f_1,\ldots,f_m)>\lambda_R\big\}.
\]
Therefore
\[
\big\|T^{K_R}(f_1,\ldots,f_m)\big\|_{1,\infty}
\geq
\lambda_R |Q_R'|.
\]
Since
\[
|Q_R'|=\left(\frac R2\right)^n,
\]
we obtain
\begin{equation}\label{eq:output-lower-full-kernel-counter}
\big\|T^{K_R}(f_1,\ldots,f_m)\big\|_{1,\infty}
\geq
c \, R^n \log(R/8),
\end{equation}
for some constant $c$ depending only on $n$ and $m$.

If~\eqref{eq:full-kernel-weak-L1-failure} were true with a constant $C$
independent of $R$, then by~\eqref{eq:KR-weak-L1} and
\eqref{eq:input-product-full-kernel-counter} we would have
\[
\big\|T^{K_R}(f_1,\ldots,f_m)\big\|_{1,\infty}
\leq
C \, \omega_D \, R^n.
\]
This contradicts~\eqref{eq:output-lower-full-kernel-counter} as
$R\to\infty$ and finishes the proof.
\end{proof}

\subsection{Proof of Theorem~\ref{thm:lorentz-zero-kernel}}

Let $K$ be a nonnegative measurable function on $(\R^n)^m$ that is nonzero on a set of positive measure. We shall prove
\eqref{eq:primary-zero-kernel-condition} and
\eqref{eq:secondary-zero-kernel-condition}.

Choose $R,M>0$ such that
\[
K_0(Y)
:=
\min\{K(Y),M\}\chi_{B(0,R)}(Y), \qquad Y \in (\R^n)^m,
\]
satisfies
\[
c_0
:=
\int_{(\R^n)^m}K_0(Y)\,\rd Y
>
0,
\]
where $B(0,R)$ denotes the ball in $(\R^n)^m$ of radius $R$ centered at the origin. 

By positivity,
\[
0\leq
T^{K_0}(f_1,\ldots,f_m)
\leq
T^K(f_1,\ldots,f_m)
\]
for nonnegative functions. Hence~\eqref{eq:lorentz-zero-kernel-boundedness}
also holds with $K_0$ in place of $K$ and with the same constant $C_K$.

Set
\[
Q:=[-2R,2R]^n,
\qquad
Q':=[-R,R]^n.
\]
If $x\in Q'$ and $Y=(y_1,\ldots,y_m)\in B(0,R)$, then $x-y_i\in Q$ for every $i=1,\ldots,m$. Therefore
\begin{equation}\label{eq:nonzero-output-block}
T^{K_0}(\chi_Q,\ldots,\chi_Q)(x)
=
c_0,
\qquad
x\in Q'.
\end{equation}

For each $N\in\N$, choose points
\[
x_1,\ldots,x_N\in\R^n
\]
so far apart that the sets
\[
Q+x_j,
\qquad
j=1,\ldots,N,
\]
are pairwise disjoint. Since $Q'\subseteq Q$, the sets
\[
Q'+x_j,
\qquad
j=1,\ldots,N,
\]
are also pairwise disjoint.

Now define
\[
F^{(N)}
:=
\sum_{j=1}^N\chi_{Q+x_j}.
\]
By positivity and multilinearity,
\[
T^{K_0}
(F^{(N)},\ldots,F^{(N)})
\geq
\sum_{j=1}^N
T^{K_0}
(\chi_{Q+x_j},\ldots,\chi_{Q+x_j}).
\]
By translation invariance and~\eqref{eq:nonzero-output-block},
\[
T^{K_0}
(\chi_{Q+x_j},\ldots,\chi_{Q+x_j})
\geq
c_0\chi_{Q'+x_j},
\qquad
j=1,\ldots,N.
\]
Consequently,
\[
T^{K_0}
(F^{(N)},\ldots,F^{(N)})
\geq
c_0
\sum_{j=1}^N\chi_{Q'+x_j}.
\]
Taking the $L^{p,q}$ quasi-norm, we obtain
\[
c_0
\left\|
\sum_{j=1}^N\chi_{Q'+x_j}
\right\|_{p,q}
\leq
\big\|
T^{K_0}
(F^{(N)},\ldots,F^{(N)})
\big\|_{p,q}.
\]
Since the sets $Q'+x_j$, $j=1,\ldots,N$, are pairwise disjoint,
\[
\sum_{j=1}^N\chi_{Q'+x_j}
=
\chi_{\bigcup_{j=1}^N(Q'+x_j)},
\]
and the union has measure $N|Q'|$. Therefore
\[
\left\|
\sum_{j=1}^N\chi_{Q'+x_j}
\right\|_{p,q}
=
c_{p,q}|Q'|^{1/p}N^{1/p}.
\]
It follows that
\[
c_0c_{p,q}|Q'|^{1/p}N^{1/p}
\leq
\big\|
T^{K_0}
(F^{(N)},\ldots,F^{(N)})
\big\|_{p,q}.
\]
On the other hand,~\eqref{eq:lorentz-zero-kernel-boundedness} gives
\[
\big\|
T^{K_0}
(F^{(N)},\ldots,F^{(N)})
\big\|_{p,q}
\leq
C_K
\prod_{i=1}^m
\|F^{(N)}\|_{p_i,q_i}.
\]
Since the sets $Q+x_j$, $j=1,\ldots,N$, are pairwise disjoint, for each
$i=1,\ldots,m$ we have
\[
\|F^{(N)}\|_{p_i,q_i}
=
c_{p_i,q_i}|Q|^{1/p_i}N^{1/p_i}.
\]
Combining the preceding estimates yields
\[
c_0c_{p,q}|Q'|^{1/p}N^{1/p}
\leq
C_K
\left(
\prod_{i=1}^m
c_{p_i,q_i}|Q|^{1/p_i}
\right)
N^{\sum_{i=1}^m1/p_i}.
\]
All factors other than the powers of $N$ are independent of $N$. Absorbing
them into a positive constant, we obtain
\[
N^{1/p}
\leq
C
N^{\sum_{i=1}^m1/p_i}.
\]
Since the constant is independent of $N$ and the preceding inequality holds
for every $N\in\N$, letting $N\to\infty$ yields
\eqref{eq:primary-zero-kernel-condition}.

Suppose now that
\[
\frac1p
=
\sum_{i=1}^m\frac1{p_i}.
\]
For $i=1,\ldots,m$, define
\[
G_i^{(N)}
:=
\sum_{j=1}^N
j^{-1/p_i}\chi_{Q+x_j}.
\]
By positivity and multilinearity,
\begin{align*}
T^{K_0}
(G_1^{(N)},\ldots,G_m^{(N)})
&\geq
\sum_{j=1}^N
j^{-\sum_{i=1}^m1/p_i}
T^{K_0}
(\chi_{Q+x_j},\ldots,\chi_{Q+x_j}).
\end{align*}
Using translation invariance and~\eqref{eq:nonzero-output-block}, we obtain
\[
T^{K_0}
(G_1^{(N)},\ldots,G_m^{(N)})
\geq
c_0
\sum_{j=1}^N
j^{-\sum_{i=1}^m1/p_i}
\chi_{Q'+x_j}.
\]
Since
\[
\sum_{i=1}^m\frac1{p_i}
=
\frac1p,
\]
this becomes
\begin{equation}\label{eq:weighted-output-lower}
T^{K_0}
(G_1^{(N)},\ldots,G_m^{(N)})
\geq
c_0
\sum_{j=1}^N
j^{-1/p}\chi_{Q'+x_j}.
\end{equation}
The sets $Q'+x_j$, $j=1,\ldots,N$, are pairwise disjoint and have the same
measure. Hence the decreasing rearrangement of
\[
\sum_{j=1}^N j^{-1/p}\chi_{Q'+x_j}
\]
is equal to $j^{-1/p}$ on the interval
\[
[(j-1)|Q'|,j|Q'|),
\qquad
j=1,\ldots,N.
\]
It follows from the definition of the Lorentz quasi-norm that there exists a
positive constant $c=c(p,q)$, independent of $N$, such that
\[
\left\|
\sum_{j=1}^N
j^{-1/p}\chi_{Q'+x_j}
\right\|_{p,q}
\geq
c\,|Q'|^{1/p}
\big(\!\log(N+1)\big)^{1/q}.
\]
Taking the $L^{p,q}$ quasi-norm in~\eqref{eq:weighted-output-lower}, we obtain
\[
c_0c\,|Q'|^{1/p}
\big(\!\log(N+1)\big)^{1/q}
\leq
\big\|
T^{K_0}
(G_1^{(N)},\ldots,G_m^{(N)})
\big\|_{p,q}.
\]
On the other hand,~\eqref{eq:lorentz-zero-kernel-boundedness} gives
\[
\big\|
T^{K_0}
(G_1^{(N)},\ldots,G_m^{(N)})
\big\|_{p,q}
\leq
C_K
\prod_{i=1}^m
\|G_i^{(N)}\|_{p_i,q_i}.
\]
For each $i=1,\ldots,m$, a similar computation with the disjoint cubes
$Q+x_j$ gives
\[
\|G_i^{(N)}\|_{p_i,q_i}
\leq
C_i\,|Q|^{1/p_i}
\big(\!\log(N+1)\big)^{1/q_i},
\]
where $C_i$ is independent of $N$. Combining the preceding estimates and
absorbing all factors independent of $N$ into the constant, we obtain
\[
\big(\!\log(N+1)\big)^{1/q}
\leq
C
\big(\!\log(N+1)\big)^{\sum_{i=1}^m1/q_i}.
\]
Since the constant is independent of $N$ and the preceding inequality holds
for every $N\in\N$, letting $N\to\infty$ yields~\eqref{eq:secondary-zero-kernel-condition} and completes the proof.

\qed

\medskip

Taking
\[
q_i=p_i,
\qquad i=1,\ldots,m,
\qquad
q=p,
\]
in Theorem~\ref{thm:lorentz-zero-kernel} gives the corresponding Lebesgue
result. In particular, if $K$ is nonnegative,
\[
\sum_{i=1}^m\frac1{p_i}<\frac1p
\]
and
\[
T^K:
L^{p_1}(\R^n)\times\cdots\times L^{p_m}(\R^n)
\to
L^p(\R^n)
\]
is bounded, then $K=0$ almost everywhere.

\section{One-parameter estimates with target exponent \texorpdfstring{$p<1$}{p<1}}
\label{sec:counterexamples}

In this section, we study the one-parameter model in the quasi-Banach range. We first prove Theorem~\ref{thm:counterexamples} by means of an additive-combinatorial construction based on finite sets $P,Q\subseteq\R$ for which
\[
|P+Q|^{1+\beta}
\]
is large compared with
\[
|P|\,|Q|\,|P-Q|^\beta.
\]
Here and throughout this section, $|P|$ denotes the cardinality of a finite set $P$.

 We conclude by recording complementary positive estimates that become
available when additional structure is imposed on the kernel. 

\subsection{A counting lemma} 

We first isolate a simple calculation for finite sets that will be used below.

\begin{lemma}\label{lem:failure_restricted_aux}
Let $N\geq2$ and set
\[
P_N:=\big\{2^0,2^1,\ldots,2^{N-1}\big\},
\qquad
Q_N:=-P_N.
\]
Then
\[
|P_N+Q_N|=N(N-1)+1,
\]
and
\[
|P_N-Q_N|=\frac{N(N+1)}2.
\]
\end{lemma}

\begin{proof}
Since $Q_N=-P_N$, we have
\[
P_N+Q_N=P_N-P_N=
\big\{2^i-2^j:\, 0\leq i,j\leq N-1\big\}.
\]
The value $0$ occurs when $i=j$. We claim that all strictly positive values
\[
2^i-2^j,
\qquad
0\leq j<i\leq N-1,
\]
are distinct. Suppose
\[
2^i-2^j=2^k-2^\ell,
\qquad
i>j,
\qquad
k>\ell.
\]
Then
\[
2^j(2^{i-j}-1)=2^\ell(2^{k-\ell}-1).
\]
Since $2^{i-j}-1$ and $2^{k-\ell}-1$ are odd, it follows that
$2^j$ divides $2^\ell$ and that $2^\ell$ divides $2^j$. Hence $j=\ell$, and then also $i=k$. There are exactly $N(N-1)/2$ pairs $(i,j)$ with $0\leq j<i\leq N-1$. Hence there are
$N(N-1)/2$ strictly positive elements of $P_N+Q_N$. By symmetry, it has the same number of negative elements. Together with the single value $0$,
we obtain
\[
|P_N+Q_N|
=
\frac{N(N-1)}2+\frac{N(N-1)}2+1
=
N(N-1)+1.
\]

Similarly,
\[
P_N-Q_N=P_N+P_N
=
\big\{2^i+2^j:\,0\leq i,j\leq N-1\big\}.
\]
We claim that the values
\[
2^i+2^j,
\qquad
0\leq i\leq j\leq N-1,
\]
are all distinct. Suppose
\[
2^i+2^j=2^k+2^\ell,
\qquad
i\leq j,
\qquad
k\leq \ell.
\]
If $j>\ell$, then
\[
2^k+2^\ell\leq 2^\ell+2^\ell=2^{\ell+1}\leq 2^j,
\]
whereas
\[
2^i+2^j>2^j,
\]
a contradiction. Similarly, $\ell>j$ is impossible. Hence $j=\ell$, which implies $2^i=2^k$, and therefore $i=k$. Thus the sums are distinct. There are exactly $N(N+1)/2$ pairs $(i,j)$ with $0\leq i\leq j\leq N-1$. Therefore
\[
|P_N-Q_N|
=
\frac{N(N+1)}2.
\]
This completes the proof.
\end{proof}

We now derive the consequence of Lemma~\ref{lem:failure_restricted_aux} that
will be used in the proof of Theorem~\ref{thm:counterexamples}. For $\beta>0$, set
\[
R_N(\beta)
:=
\frac{|P_N+Q_N|^{1+\beta}}
{|P_N|\,|Q_N|\,|P_N-Q_N|^\beta}.
\]
By Lemma~\ref{lem:failure_restricted_aux}, we obtain 
\[
R_N(\beta)
=
\frac{(N(N-1)+1)^{1+\beta}}
{N^2\left(\frac{N(N+1)}2\right)^\beta}.
\]
Writing
\[
N(N-1)+1
=
N^2\left(1-\frac1N+\frac1{N^2}\right)
\]
and
\[
\frac{N(N+1)}2
=
\frac{N^2}{2}\left(1+\frac1N\right),
\]
we get
\[
R_N(\beta)
=
2^\beta
\frac{
\left(1-\frac1N+\frac1{N^2}\right)^{1+\beta}
}
{
\left(1+\frac1N\right)^\beta
}.
\]
Therefore
\[
R_N(\beta)\to 2^\beta
\qquad
\text{as } N\to\infty.
\]
Since $2^\beta>1$, there exists $N=N(\beta)\geq2$ such that
\begin{equation}\label{eq:beta-base-defect}
R_N(\beta)
=
\frac{|P_N+Q_N|^{1+\beta}}
{|P_N|\,|Q_N|\,|P_N-Q_N|^\beta}
>1.
\end{equation}
Indeed, if
\[
\delta:=\frac{2^\beta-1}{2}>0,
\]
then for all sufficiently large $N$ we have
\[
R_N(\beta)>2^\beta-\delta=\frac{2^\beta+1}{2}>1.
\]

\subsection{Proof of Theorem~\ref{thm:counterexamples}}
Let $N=N(\beta)$ be chosen so that~\eqref{eq:beta-base-defect} holds, and set
\[
P:=P_N,
\qquad
Q:=Q_N,
\]
where $P_N$ and $Q_N$ are the sets from Lemma~\ref{lem:failure_restricted_aux}.
Thus
\begin{equation}\label{eq:fixed-base-defect}
\frac{|P+Q|^{1+\beta}}
{|P|\,|Q|\,|P-Q|^\beta}
>1.
\end{equation}

For each $k\in\N$, set
\[
\theta_j:=\pi^j,
\qquad
j=0,\ldots,k-1.
\]
Since $\pi$ is transcendental, the numbers
\[
\theta_0,\theta_1,\ldots,\theta_{k-1}
\]
are linearly independent over $\Q$. Define
\[
\Phi_k(x_0,\ldots,x_{k-1})
:=
\sum_{j=0}^{k-1}x_j\theta_j.
\]
Since the numbers $\theta_j$ are linearly independent over $\Q$, the map
$\Phi_k$ is injective on every subset of $\Z^k$.

We now set
\[
P^{(k)}:=\Phi_k(P^k),
\qquad
Q^{(k)}:=\Phi_k(Q^k),
\]
where $P^k$ denotes the set of all $k$-tuples with entries in $P$, namely
\[
P^k
=
\big\{(p_0,\ldots,p_{k-1}):\, p_j\in P,\ j=0,\ldots,k-1\big\},
\]
and similarly for $Q^k$. 
Hence
\[
|P^{(k)}|=|P^k|=|P|^k,
\qquad
|Q^{(k)}|=|Q^k|=|Q|^k.
\]
Moreover, by linearity of $\Phi_k$,
\[
P^{(k)}+Q^{(k)}
=
\Phi_k\big((P+Q)^k\big),
\]
and therefore
\[
|P^{(k)}+Q^{(k)}|
=
|(P+Q)^k|
=
|P+Q|^k.
\]
Similarly,
\[
P^{(k)}-Q^{(k)}
=
\Phi_k((P-Q)^k),
\]
and so
\[
|P^{(k)}-Q^{(k)}|
=
|(P-Q)^k|
=
|P-Q|^k.
\]

Next, fix $k\in\N$. Since the sets
\[
P^{(k)},
\qquad
Q^{(k)},
\qquad
\frac12\big(P^{(k)}-Q^{(k)}\big),
\qquad
\frac12\big(P^{(k)}+Q^{(k)}\big)
\]
are finite, we may choose $\eta>0$ so small that the intervals centered at
distinct points of each of these sets are pairwise disjoint. Write
\[
I_\eta:=(-\eta,\eta).
\]
Thus, for instance,
\[
P^{(k)}+I_\eta
=
\bigcup_{p\in P^{(k)}}(p-\eta,p+\eta),
\]
and the intervals in this union are pairwise disjoint. Define
\[
E_1:=P^{(k)}+I_\eta,
\qquad
E_2:=Q^{(k)}+I_\eta,
\]
and
\[
A:=\frac12\big(P^{(k)}-Q^{(k)}\big)+I_\eta.
\]
Then
\[
|E_1|=2\eta |P^{(k)}|=2\eta |P|^k,
\]
\[
|E_2|=2\eta |Q^{(k)}|=2\eta |Q|^k,
\]
and
\[
|A|=2\eta |P^{(k)}-Q^{(k)}|=2\eta |P-Q|^k.
\]

Now set
\[
H(x):=
\int_{\R}
\chi_{E_1}(x+y) \, \chi_{E_2}(x-y) \, \chi_A(y)\,\dy.
\]
We obtain a lower bound for a level set of $H$. Let
\[
s\in P^{(k)}+Q^{(k)}.
\]
Choose one representation
\[
s=p+q,
\qquad
p\in P^{(k)},
\qquad
q\in Q^{(k)}.
\]
If
\[
x\in \frac{s}{2}+I_{\eta/4}
\]
and
\[
y\in \frac{p-q}{2}+I_{\eta/4},
\]
then
\[
x+y\in p+I_\eta\subseteq E_1, \qquad 
x-y\in q+I_\eta\subseteq E_2,
\]
and
\[
y\in \frac12\big(P^{(k)}-Q^{(k)}\big)+I_\eta=A.
\]
Thus, for every
\[
x\in \frac{s}{2}+I_{\eta/4},
\]
we have
\[
H(x)
\geq
\int_{\frac{p-q}{2}+I_{\eta/4}}1\,\rd y
=
|I_{\eta/4}|
=
\frac{\eta}{2}.
\]
Consequently,
\[
H(x)>\frac{\eta}{4}
\]
on the set
\[
S_k
:=
\bigcup_{s\in P^{(k)}+Q^{(k)}}
\left(
\frac{s}{2}+I_{\eta/4}
\right).
\]
By the choice of $\eta$, the intervals in this union are disjoint. Therefore
\[
|S_k|
=
\frac{\eta}{2}|P^{(k)}+Q^{(k)}|
=
\frac{\eta}{2}|P+Q|^k.
\]
Since
\[
S_k\subseteq \{H>\eta/4\},
\]
the definition of the weak $L^{1/(1+\beta)}$ quasi-norm gives
\begin{align*}
\|H\|_{1/(1+\beta),\infty}
&\geq
\frac{\eta}{4}|S_k|^{1+\beta} \\
&=
\frac{\eta}{4}
\left(
\frac{\eta}{2}|P+Q|^k
\right)^{1+\beta} \\
&=
\frac{\eta^{2+\beta}}{2^{3+\beta}}
|P+Q|^{k(1+\beta)}.
\end{align*}
On the other hand,
\begin{align*}
|E_1|\,|E_2|\,|A|^\beta
&=
(2\eta |P|^k)(2\eta |Q|^k)
(2\eta |P-Q|^k)^\beta \\
&=
(2\eta)^{2+\beta}
\big(|P|\,|Q|\,|P-Q|^\beta\big)^k.
\end{align*}
Consequently,
\[
\frac{\|H\|_{1/(1+\beta),\infty}}
{|E_1|\,|E_2|\,|A|^\beta}
\geq
\frac{1}{2^{5+2\beta}}
\left(
\frac{|P+Q|^{1+\beta}}
{|P|\,|Q|\,|P-Q|^\beta}
\right)^k.
\]
By~\eqref{eq:fixed-base-defect}, the factor in parentheses is strictly larger than $1$. Letting $k\to\infty$, the ratio tends to infinity. Hence no constant $C$ can make~\eqref{eq:counterexamples} hold for all measurable sets $E_1,E_2,A\subseteq\R$ of finite measure.
\qed

\subsection{Related positive results for structured kernels}

 Theorem~\ref{thm:counterexamples} shows that a natural family of restricted weak-type estimates with target exponent $p<1$ cannot hold uniformly over arbitrary measurable kernel sets, even in the one-dimensional bilinear case. This does not rule out estimates for fixed kernels or for classes of kernels with additional structure.

We now pass to the corresponding bilinear one-parameter operator on $\R^n$,
obtained from~\eqref{eq:one_parameter_operator} by taking $m=2$,
$\alpha_1=-1$, and $\alpha_2=1$, and record two complementary positive results.

First, we consider cube kernels. By a \emph{cube} we mean a half-open cube
with sides parallel to the coordinate axes. A direct adaptation of the
argument in~\cite[Lemma~4.1]{alves2026bilinear} gives the strong endpoint
estimate
\[
L^1\times L^1\longrightarrow L^{1/2}.
\]

\begin{proposition}\label{prop:cube}
There exists a positive constant $C_n$, depending only on $n$, such that
\begin{equation}\label{eq:cube}
\big\|T_{\chi_Q}(f_1,f_2)\big\|_{1/2}
\leq
C_n \, |Q| \, \|f_1\|_1 \, \|f_2\|_1,
\end{equation}
for every cube
$Q\subseteq\R^n$ and all integrable functions $f_1,f_2$ on $\R^n$.
\end{proposition}

Next, we consider weak kernels with radial monotonicity. In this case, the weak
$L^r$ condition gives a pointwise estimate for the kernel, and this allows us to
dominate $T_g$ by the bilinear fractional integral
\begin{equation}\label{eq:B-alpha-definition}
B_\alpha(f_1,f_2)(x)
:=
\int_{\R^n}
f_1(x+y) \, f_2(x-y) \,|y|^{\alpha-n}\,\rd y,
\qquad 0 < \alpha < n.
\end{equation}

We say that a nonnegative measurable function $g$ on $\R^n$ is \emph{radially decreasing} if there exists a nonincreasing function $\psi:(0,\infty)\to[0,\infty)$ such that
\[
g(y)=\psi(|y|) \qquad \text{for a.e.\ } y \in \R^n.
\]

\begin{lemma}\label{lem:radial-weak-pointwise}
Let $1<r<\infty$, and let $g$ be a radially decreasing function on $\R^n$.
If $g\in L^{r,\infty}(\R^n)$, then there exists a positive constant $C=C(n,r)$ such that 
\begin{equation}\label{eq:radial-weak-pointwise}
g(y) \leq C \, \|g\|_{r,\infty} \, |y|^{-n/r} \qquad \text{for a.e.\ } y \in \R^n.
\end{equation}
\end{lemma}

\begin{proof}
After modifying $g$ on a set of measure zero, we may assume that
$
g(y)=\psi(|y|)
$
for every $y\neq0$, where $\psi$ is nonincreasing. Fix $y\neq0$ and note that if $0<\lambda<g(y)$, then 
\[
g(z)\geq g(y)>\lambda
\qquad
\text{whenever } |z|\leq |y|.
\]
Therefore
$
B(0,|y|)\subseteq \big\{ g>\lambda\big\}$,
and hence
\[
c_n|y|^n
=
|B(0,|y|)|
\leq
|\{g>\lambda\}|.
\]
It follows that
\[
\lambda
(c_n|y|^n)^{1/r}
\leq
\lambda |\{g>\lambda\}|^{1/r}
\leq
\|g\|_{r,\infty},
\]
for all $0 < \lambda < g(y)$. Letting $\lambda\to g(y)$ gives~\eqref{eq:radial-weak-pointwise}.
\end{proof}

For $1<r<\infty$, define
\begin{equation}\label{eq:alpha-from-r}
\alpha
:=
n\left(1-\frac1r\right).
\end{equation}
Then $0<\alpha<n$ and
\[
|y|^{-n/r}
=
|y|^{\alpha-n}.
\]

Let $\mathcal P_\alpha$ denote the interior of the pentagon in $[0,1]^2$ with
vertices
\[
(1,1),
\qquad
(1,0),
\qquad
(\alpha/n,0),
\qquad
(0,\alpha/n),
\qquad
(0,1).
\]
We use the known bounds for $B_\alpha$; see
\cite{kenig1999multilinear,grafakos2001remarks,alves2026uniform}. If
\[
\left(\frac1{p_1},\frac1{p_2}\right)\in\mathcal P_\alpha
\qquad\text{and}\qquad
\frac1p+\frac{\alpha}{n}
=
\frac1{p_1}+\frac1{p_2},
\]
then there exists a positive constant $C=C(\alpha,n,p_1,p_2)$ such that
\[
\big\|B_\alpha(f_1,f_2)\big\|_p
\leq
C \,
\|f_1\|_{p_1}
\|f_2\|_{p_2},
\]
for all $f_i\in L^{p_i}(\R^n)$, $i=1,2$.

Combining Lemma~\ref{lem:radial-weak-pointwise} with these bounds for $B_\alpha$ gives the following immediate consequence.

\begin{proposition}\label{prop:radial-weak-bounds}
Let $1<r<\infty$, let $\alpha$ be given by~\eqref{eq:alpha-from-r}, and let
$g$ be a nonnegative radially decreasing function in $L^{r,\infty}(\R^n)$.
Suppose that
\[
\left(\frac1{p_1},\frac1{p_2}\right)\in\mathcal P_\alpha
\]
and define $0 < p < \infty$ by
\[
\frac1p+\frac{\alpha}{n}
=
\frac1{p_1}+\frac1{p_2}.
\]
Then there exists a positive constant
$C=C(n,r,p_1,p_2)$ such that
\begin{equation}\label{eq:radial-weak-bounds}
\big\|T_g(f_1,f_2)\big\|_p
\leq
C \,
\|f_1\|_{p_1} \,
\|f_2\|_{p_2} \, \|g\|_{r,\infty}
\end{equation}
for all $f_1\in L^{p_1}(\R^n)$ and $f_2\in L^{p_2}(\R^n)$.
\end{proposition}

Since $\mathcal P_\alpha$ contains points for which
\[
\frac1{p_1}+\frac1{p_2}-\frac{\alpha}{n}>1,
\]
Proposition~\ref{prop:radial-weak-bounds} includes estimates with $p<1$.
Thus radial monotonicity gives an additional condition on the kernel under which
bounds in the quasi-Banach range may hold.


\end{document}